\documentclass[11pt]{amsart}
\usepackage{amscd,amssymb,longtable, rotating, lscape, graphicx}
\usepackage[matrix,arrow,curve]{xy}
\usepackage{supertabular}
\usepackage{setspace}
\theoremstyle{definition}
\newtheorem{theorem}[equation]{Theorem}
\newtheorem*{theorem*}{Theorem}
\newtheorem{lemma}[equation]{Lemma}
\newtheorem{corollary}[equation]{Corollary}

\newtheorem{example}[equation]{Example}

\newtheorem*{definition*}{Definition}
\newtheorem{question}[equation]{Question}

\theoremstyle{remark}

\makeatletter\@addtoreset{equation}{section}
\makeatletter\@addtoreset{section}{part}

\makeatother

\def \OO {\mathcal{O}}

\def \P {\mathbb{P}}

\def \C {\mathbb{C}}
\def \Z {\mathbb{Z}}

\def \GL {\mathrm{GL}\,}
\def \PGL {\mathrm{PGL}\,}
\def \SL {\mathrm{SL}\,}

\def \Sym {\mathrm{Sym}}

\def \MM {\mathrm{M}}
\def \J {\mathrm{J}}
\def \HaJ {\mathrm{J}_2}
\def \Suz {\mathrm{Suz}}
\def \HS {\mathrm{HS}}
\def \Co {\mathrm{Co}}
\def \McL {\mathrm{McL}}
\def \Fi {\mathrm{Fi}}
\def \He {\mathrm{He}}
\def \HN {\mathrm{HN}}
\def \Th {\mathrm{Th}}
\def \B {\mathrm{B}}
\def \M {\mathrm{M}}
\def \T {\mathrm{T}}
\def \Ru {\mathrm{Ru}}
\def \Ly {\mathrm{Ly}}
\def \ON {\mathrm{O}\mbox{'}\mathrm{N}}

\def \ge {\geqslant}
\def \le {\leqslant}

\author{Ivan Cheltsov and Constantin Shramov}

\title{Sporadic simple groups and quotient singularities}

\address{University of Edinburgh, Edinburgh EH9 3JZ, UK}

\address{Steklov Institute of Mathematics, Moscow 119991, Russia}

\address{\texttt{I.Cheltsov@ed.ac.uk}\ \ \texttt{shramov@mccme.ru}}

\dedicatory{to Igor Rostislavovich Shafarevich with deep respect}

\begin{document}

\begin{abstract}
We show that the only sporadic simple group such that 
some faithful representation of some of its
stem extensions gives rise to  exceptional (weakly-exceptional but
not exceptional, respectively) quotient singularities is the
Hall--Janko group (the Suzuki group, respectively).
\end{abstract}

\maketitle

\section{Introduction}
\label{section:intro}

Finite subgroups in $\SL_2(\C)$ have been classified more than a hundred 
years ago. The quotients of~$\C^2$ by these groups are
$\mathbb{A}$-$\mathbb{D}$-$\mathbb{E}$ singularities, which are
also known by other names (Kleinian singularities, Du Val
singularities, rational surface double points, two-dimensional
canonical singularities etc). Shokurov suggested a higher
dimensional generalization of the singularities of type
$\mathbb{E}$ and of both types $\mathbb{D}$ and $\mathbb{E}$. He
called them exceptional and weakly-exceptional, respectively. The
precise definitions of exceptional and weakly-exceptional
singularities are quite technical (see
\cite[Definition~1.5]{Sho00} and \cite[Definition~4.1]{Pr98plt},
respectively). Surprisingly, they are connected with a wide range
of algebraic and geometric questions.

It turned out that exceptional and weakly-exceptional
singularities are related to the Calabi problem for orbifolds with
positive first Chern class (see \cite{ChSh09}).

\begin{example}
\label{example:intro-quasi-homogeneous} Let $(V\ni O)$ be a germ
of three-dimensional isolated quasihomogeneous hypersurface
singularity that is given by
$$
\phi\big(x,y,z,t\big)=0\subset\mathbb{C}^{4}\cong\mathrm{Spec}\Big(\mathbb{C}\big[x,y,z,t\big]\Big),
$$
where $\phi(x,y,z,t)$ is a quasihomogeneous polynomial of degree
$d$ with respect to some weights $\mathrm{wt}(x)=a_{0}$,
$\mathrm{wt}(y)=a_{1}$, $\mathrm{wt}(z)=a_{2}$,
$\mathrm{wt}(t)=a_{3}$ such that $a_0\leqslant a_1\leqslant
a_2\leqslant a_3$ and $\mathrm{gcd}(a_{0},a_{1},a_{2},a_{3})=1$.
Let $S$ be a~weighted hypersurface in
$\mathbb{P}(a_{0},a_{1},a_{2},a_{3})$ of degree $d$ that is given
by the same equation $\phi(x,y,z,t)=0$. Suppose, in addition, that
$\sum_{i=0}^{n}a_{i}>d$, and  $S$ is well-formed (see
\cite[Definition~6.9]{IF00}). Then $S$ is a Del Pezzo surface with
at most quotient singularities. Moreover, if $(V\ni O)$ is either
exceptional or weakly-exceptional, then $S$ admits an orbifold
Kahler--Einsten metric (this follows, for example, from
\cite{Ti87}, \cite[Theorem~4.9]{Pr98plt},
\cite[Theorem~2.1]{Kud01}, and  \cite[Theorem~A.3]{ChSh08c}).
\end{example}

Many old and still open group-theoretic questions have
algebro-geometric counterparts related to the exceptionality of
quotient singularities (see, for example, \cite{Th81} and
\cite[Conjecture~1.25]{ChSh09}). It seems that the study of
exceptionality and weak-exceptionality of quotient singularities
may shed new light on some group-theoretic problems.

\begin{example}
\label{example:exceptional-quotient} Let $G$ be a finite subgroup
in $\GL_{n+1}(\C)$ that does not contain
reflections\footnote{Recall that an element $g\in\GL_{n+1}(\C)$ is called a
\emph{reflection} (or sometimes a \emph{quasi-reflection}) if it
has exactly one eigenvalue that is different from $1$.}, and let
$G^{\prime}$ be a finite subgroup in $\GL_{n+1}(\C)$ that does not
contain reflections such that $G^{\prime}$ and $G$ has the same
image in $\PGL_{n+1}(\C)$. Then it follows from
\cite[Theorem~3.15]{ChSh09} and \cite[Theorem~3.16]{ChSh09} that
the singularity $\C^{n+1}\slash G$ is exceptional
(weakly-exceptional, respectively) if and only if the singularity
$\C^{n+1}\slash G^{\prime}$ is exceptional (weakly-exceptional,
respectively). Moreover, it follows from
\cite[Theorem~1.30]{ChSh09} and \cite[Theorem~3.15]{ChSh09} that
the subgroup $G\subset\GL_{n+1}(\C)$ is transitive (i.\,e. the
corresponding $(n+1)$-dimensional representation is irreducible)
provided that the singularity~\mbox{$\C^{n+1}\slash G$} is
weakly-exceptional. Similarly, it follows from
\cite[Theorem~1.29]{ChSh09} that $G$ must be primitive (see
\cite{Bli17} or, for example, \cite[Definition~1.21]{ChSh09}) if
$\C^{n+1}/G$ is exceptional. Finally, it follows from
\cite[Theorem~3.16]{ChSh09} (\cite[Theorem~3.15]{ChSh09},
respectively) that $\C^{n+1}\slash G$ is not exceptional if $G$
has a~semi-invariant\footnote{Recall that a \emph{semi-invariant}
of degree $d$ of the group $G\subset\GL_{n+1}(\C)$ is a one-dimensional
subrepresentation in $\mathrm{Sym}^{d}(\mathbb{C}^{n+1})$.} of
degree~at~most~$n+1$ (of degree~at~most~$n$, respectively).%
\end{example}

Starting from this point, we restrict ourselves to the case of
quotient singularities.

In low dimensions, the study of exceptional and weakly-exceptional
quotient singularities is closely related to the classification of
finite collineation groups (see \cite{Bli17}, \cite{Br67},
\cite{Li71}, \cite{Wa69}, \cite{Wa70}, \cite{Fe71}). Using
classical results of Blichfeldt, Brauer, and Lindsey, exceptional
quotient singularities of dimensions $3$, $4$, $5$ and $6$ have
been completely classified by Markushevich, Prokhorov, and the
authors (see \cite{MarPr99}, \cite{ChSh09}, \cite{ChSh10}).
Moreover, we used the classification obtained by Wales in
\cite{Wa69} and~\cite{Wa70} to prove that seven-dimensional
exceptional quotient singularities do not exist (see
\cite{ChSh10}). Sakovics classified weakly-exceptional quotient
singularities of dimensions $3$ and $4$ (see \cite{Sak10}).
Higher-dimensional weakly-exceptional quotient singularities were
studied in \cite{Crelle}. Unfortunately, we have no clear picture
which finite subgroups in $\GL_{n+1}(\C)$ give rise to exceptional
or weakly-exceptional singularities for $n\gg 0$.

A surprising fact observed in \cite{ChSh10} is that among the
(very few) groups corresponding to exceptional six-dimensional
quotient singularities there appears a central extension $2.\HaJ$
of the Hall--Janko sporadic simple group (see~\cite{Li70}).
Actually, this property is very rare for the projective
representations of sporadic simple groups, so that essentially we
have only one more example of this kind of behavior among them. It
is related to the Suzuki sporadic simple group (see
\cite{Suzuki69}). In this paper we prove the following

\begin{theorem}[{cf.~\cite[Theorem~14]{Gokova}}]
\label{theorem:main} Let $G$ be a sporadic simple finite group, or
its stem extension.\footnote{Recall that a \emph{stem extension}
of a perfect group $G$ is a central extension of $G$ that can be
obtained as a quotient of the universal central extension of $G$
by its Schur multiplier, see e.\,g.~\cite[\S9.4]{Rotman} for
detailes.} Let $G\hookrightarrow\GL(U)$ be a (faithful) finite
dimensional complex representation of $G$. Then the singularity
$U/G$ is exceptional if and only if $G\cong 2.\HaJ$, and $U$ is a
$6$-dimensional irreducible representation of $G$. The singularity
$U/G$ is weakly-exceptional but not exceptional if and only if
$G\cong 6.\Suz$ is the central extension of the Suzuki simple
group by the cyclic group of order $6$, and $U$ is a
$12$-dimensional irreducible representation of $G$.
\end{theorem}

Theorem~\ref{theorem:main} shows that the groups $\HaJ$ and $\Suz$
are somehow distinguished among the sporadic simple groups from
the geometric point of view, and therefore motivates the following

\begin{question}\label{question:WTF}
Is there some group-theoretic property that distinguishes the
groups $\HaJ$ and $\Suz$ among the sporadic simple groups?
\end{question}

As one can see from Appendix~\ref{section:semi-invariants}, one of
the characterizations of these groups comes from the fact that the
groups $2.\HaJ$ and $6.\Suz$ have irreducible representations with
no semi-invariants of low degrees. Note that this is a priori not
equivalent to weak exceptionality of the corresponding 
quotient singularity, and the geometric
characterization via weak exceptionality requires another series
of coincidences. On the other hand, it would be interesting to
know if there is some intrinsic characterization of the groups
$\HaJ$ and $\Suz$ that goes beyond the observation concerning the
semi-invariants~--- possibly not even involving representation
theory at all.
One of the goals of this paper as we see it is to
attract attention of the experts in group theory to
Question~\ref{question:WTF}, and more generally to a more broad
range of questions on the possible interplay between the
properties of certain groups and geometrical properties of the
corresponding quotient singularities.\footnote{
Note that there is an interesting characterization
of the groups $2.\HaJ$ and $6.\Suz$ together with the representations
of these groups arising in Theorem~\ref{theorem:main}
via irreducibility of symmetric powers
obtained in~\cite[Theorem~1.1]{GurTiep},
that has geometric implications concerning stable vector bundles,
cf.~\cite[Corollary~1.3]{GurTiep} and~\cite{BalajiKollar}.}

To study exceptionality and weak-exceptionality of a singularity
$\C^{n+1}/G$ for a finite subgroup $G\subset\GL_{n+1}(\C)$, one
can always assume that the group $G$ does not contain reflections
(cf. Example~\ref{example:exceptional-quotient} and
\cite[Remark~1.16]{ChSh10}). Keeping in mind
Example~\ref{example:exceptional-quotient}, we see that to prove
Theorem~\ref{theorem:main}, we may restrict ourselves to the case
of irreducible representations. Similarly, it follows from
Example~\ref{example:exceptional-quotient} that we may exclude
from our search the groups that have semi-invariants of low
degrees for the corresponding representations 
by a straightforward case by case study. The results of
the corresponding computations are listed in
Appendix~\ref{section:semi-invariants}. They were obtained using
the GAP software (see \cite{GAP}) and the classification of all
finite simple groups (see \cite{Atlas}) and communicated to us by
A.\,Zavarnitsyn. As a result, we are left with just two
candidates: the group $2.\HaJ$ acting in $U\cong\C^6$, and the
group $6.\Suz$ acting in $U\cong\C^{12}$. The exceptionality of
the quotient singularity corresponding to the first case was
settled in~\cite{ChSh10}. Therefore, the only new result of
geometric nature we obtain here is the following theorem that is
proved in Section~\ref{section:Suzuki}.

\begin{theorem}\label{theorem:Suzuki}
Let $G\cong 6.\Suz$, and let $U$ be a $12$-dimensional
irreducible representation of $G$.
Then the singularity $U/G$
is weakly-exceptional but not exceptional.
\end{theorem}

\medskip

We would like to thank A.\,Zavarnitsyn for computational support,
and P.\,H.\,Tiep for interesting discussions.
We would like to thank J.\,Park for inviting us to Pohang
Mathematics Institute where this paper was finished. Finally, we
would like to thank V.\,Przyjalkowski for creating an inspiring
atmosphere while we have been working on this paper.
The authors were partially supported by the grants
RFBR~11-01-00336-a, N.Sh.~4713.2010.1,
MK-6612.2012.1, and AG Laboratory SU-HSE, RF government
grant ag.~11.G34.31.0023.

\section{Suzuki simple group}
\label{section:Suzuki}

In this section we prove Theorem~\ref{theorem:Suzuki} using the
method we first applied in~\cite{Gokova} and the following

\begin{theorem}[{\cite[Theorem~1.12]{Crelle}}]
\label{theorem:weakly-exceptional-quotient-criterion} Let $G$ be a
finite group in $\GL_{n+1}(\C)$ that does not contain reflections,
and let $\bar{G}$ be the image of the group $G$ in
$\PGL_{n+1}(\mathbb{C})$. If $\C^{n+1}\slash G$ is not
weakly-exceptional, then there~is a $\bar{G}$-invariant Fano
type\footnote{A variety $V$ is said to be of \emph{Fano type} if
it is irreducible, normal, and there exists an effective
$\mathbb{Q}$-divisor $\Delta_{V}$ on $V$ such that
$-(K_{V}+\Delta_{V})$ is a~$\mathbb{Q}$-Cartier ample divisor, and
the log pair $(V,\Delta_{V})$ has at most Kawamata log terminal
singularities.} projectively normal subvariety $V\subset\P^n$ such
that
$$
\mathrm{deg}\big(V\big)\leqslant {n\choose \mathrm{dim}\big(V\big)},%
$$
and for every $i\geqslant 1$ and for every $m\geqslant 0$, we have
$h^{i}(\mathcal{O}_{\P^n}(m)\otimes
\mathcal{I}_{V})=h^{i}(\mathcal{O}_{V}(m))=0$, and
$$
h^{0}\Big(\mathcal{O}_{\P^n}\Big(\big(\mathrm{dim}(V)+1\big)\Big)\otimes 
\mathcal{I}_{V}\Big)\geqslant {n\choose \mathrm{dim}\big(V\big)+1},%
$$
where $\mathcal{I}_{V}$ is the ideal sheaf of the subvariety
$V\subset\P^n$. Let $\Pi$ be a general linear subspace in $\P^n$
of codimension $k\leqslant\mathrm{dim}(V)$. Put $X=V\cap\Pi$. Then
$h^{i}(\mathcal{O}_{\Pi}(m)\otimes \mathcal{I}_{X})=0$ for every
$i\geqslant 1$ and $m\geqslant k$, where $\mathcal{I}_{X}$ is the
ideal sheaf of the subvariety $X\subset\Pi$. Moreover, if $k=1$
and $\mathrm{dim}(V)\geqslant 2$, then $X$ is irreducible,
projectively normal and $h^{i}(\mathcal{O}_{X}(m))=0$ for every
$i\geqslant 1$ and $m\geqslant 1$.
\end{theorem}

Let $G\cong 6.\Suz$ be the central extension of the Suzuki
sporadic simple group. Then there is an embedding
$G\hookrightarrow\SL_{12}(\mathbb{C})$ that is given by an
irreducible $12$-dimensional $G$-representation $U$.

Denote by $\Delta_k$ the collection of dimensions of irreducible
subrepresentations of $\Sym^k(U^{\vee})$. We will use the
following notation: writing $\Delta_k=[\ldots, r\times m,\ldots]$,
we mean that there are exactly~$r$ summands
dimension~$m$ (not necessarily isomorphic to each other)
in the decomposition of~\mbox{$\Sym^k(U^{\vee})$}
into a sum of irreducible subrepresentations.
Furthermore, denote by $\Sigma_k$ the set of partial sums
of~$\Delta_k$, i.\,e. the set of all numbers $s=\sum r'_im_i$,
where 
$$\Delta_k=[r_1\times m_1,r_2\times m_2,\ldots, r_i\times
m_i,\ldots]$$ 
and $0\le r_i'\le r_i$ for all $i$. We use the
abbreviation $m_{i}$ for $1\times m_{i}$.

We will need the following properties of the $G$-representation
$U$ that can be verified by direct computations. We used the GAP
software (see \cite{GAP}) to carry them out.

\begin{lemma}\label{lemma:splitting}
The representations $\Sym^k(U^{\vee})$ are irreducible for $2\le
k\le 5$ (and have dimensions $78$, $364$, $1365$ and $4368$,
respectively). Futhermore, $\Delta_6=[364,12012]$,
$\Delta_7=[4368,27456]$,
\begin{gather*}
\Delta_8=[1365,4290,27027,42900],\quad
\Delta_9=[2\times 364, 2\times 16016, 35100,100100],\\
\Delta_{10}=[78,1365,3003,4290, 2\times 27027, 2\times 75075, 139776],\\
\Delta_{11}=[12,924, 2\times 4368, 2\times 12012,
2\times 27456, 112320,144144, 2\times 180180],
\end{gather*}
and
\begin{multline*}
\Delta_{12}=[1,143, 2\times 364, 1001, 2\times 5940, 2\times 12012,
2\times 14300, 2\times 15015, 15795, 25025,\\
2\times 40040, 54054,75075,
88452, 2\times 93555, 2\times 100100, 163800,
168960,197120].
\end{multline*}
\end{lemma}

\begin{corollary}\label{corollary:no-semiinvariants}
The group $G$ does not have semi-invariants of degree $d\leqslant
11$, and does have a semi-invariant of degree $d=12$.
\end{corollary}

Now we are ready to prove Theorem~\ref{theorem:main}. Suppose that
the singularity $U/G$ is not weakly-exceptional. Let $\bar{G}$ be
the image of the group $G$ in $\PGL_{12}(\mathbb{C})$. Then it
follows from
Theorem~\ref{theorem:weakly-exceptional-quotient-criterion} that
there~is a $\bar{G}$-invariant Fano type
projectively normal subvariety \mbox{$V\subset\P^{11}$} such that
$$
\mathrm{deg}\big(V\big)\leqslant {11\choose \mathrm{dim}\big(V\big)},%
$$
and for every $i\geqslant 1$ and $m\geqslant 0$, we have
$h^{i}(\mathcal{O}_{\P^{11}}(m)\otimes
\mathcal{I}_{V})=h^{i}(\mathcal{O}_{V}(m))=0$, where
$\mathcal{I}_{V}$ is the ideal sheaf of the subvariety
$V\subset\P^{11}$. Put $n=\dim(V)$.

\begin{lemma}\label{lemma:0-10}
One has $1\le n\le 9$.
\end{lemma}

\begin{proof}
One has $n\neq 0$ since $U$ is an irreducible representation of
the group $G$. On the other hand, if $n=10$ then $V$ is an
$\bar{G}$-invariant hypersurface such that $\deg(V)\leqslant 11$,
which contradicts Corollary~\ref{corollary:no-semiinvariants}.
\end{proof}

Put $h_m=h^0(\OO_V(m))$ and
$q_m=h^0(\OO_{\P^{11}}(m)\otimes\mathcal{I}_{V})$ for every
$m\in\mathbb{Z}$. Then
$$
q_m=h^0\Big(\OO_{\P^{11}}\big(m\big)\Big)-h_m={11+m\choose m}-h_m
$$
for every $m\geqslant 1$, since
$h^{1}(\mathcal{O}_{\P^{11}}(m)\otimes \mathcal{I}_{V})=0$ for
every $m\geqslant 0$.

Let $H$ be a general hyperplane section of $V$. Put
$d=H^n=\deg(V)$ and $H_V(m)=\chi(\OO_{V}(m))$. Then $H_V(m)=h_m$
for every $m\geqslant 1$, since $h^{i}(\mathcal{O}_V(mH))=0$ for
every $i\geqslant 1$ and every $m\geqslant 0$. Recall that
$H_V(m)$ is a Hilbert polynomial of the subvariety $V$, which is a
polynomial in $m$ of degree $n$ with leading coefficient $d/n!$.

It follows from
Theorem~\ref{theorem:weakly-exceptional-quotient-criterion} that
$V$ has one more property that we need.  Let
$\Lambda_{1},\Lambda_{2},\ldots,\Lambda_{n}$ be general
hyperplanes in $\P^{11}$. Put $\Pi_{j}=\Lambda_{1}\cap\ldots\cap
\Lambda_{j}$, $V_{j}=V\cap\Pi_{j}$, and $H_{j}=V_{j}\cap H$ for
every $j\in\{1,\ldots,n\}$. Put $V_{0}=V$, $H_{0}=H$,
$\Pi_0=\P^{11}$. For every $j\in\{0,1,\ldots,n\}$, let
$\mathcal{I}_{V_{j}}$ be the ideal sheaf of the subvariety
$V_{j}\subset\Pi_{j}$. Then it follows from
Theorem~\ref{theorem:weakly-exceptional-quotient-criterion} that
$h^{i}(\mathcal{O}_{\Pi_j}(m)\otimes \mathcal{I}_{V_j})=0$ for
every $i\geqslant 1$ and $m\geqslant j$.

Recall that $\Pi_{j}\cong\P^{11-j}$ and put
$q_i(V_j)=h^0(\mathcal{O}_{\Pi_{j}}(i)\otimes\mathcal{I}_{V_{j}})$
for every $j\in\{0,1,\ldots,n\}$.

\begin{lemma}\label{lemma:qivj}
Suppose that $i\geqslant j+1$ and $j\in\{1,\ldots,n\}$. Then
$$
q_i(V_j)=q_i-{j\choose 1}q_{i-1}+{j\choose 2}q_{i-2}-\ldots+(-1)^jq_{i-j}.%
$$
\end{lemma}

\begin{proof}
See the proof of \cite[Lemma~27]{Gokova}.
\end{proof}

Recall that $q_1=0$ since the representation $U$ is irreducible,
and $q_i=0$ for $2\le i\le 5$ by Lemma~\ref{lemma:splitting}.
Therefore, we have

\begin{corollary}\label{corollary:qivj-stupid}
If $n=9$, one has
$$
q_9-8q_8+28q_7-56q_6=q_9(V_8)\ge 0.
$$
\end{corollary}

Playing with the numbers $q_i(V_j)$, we obtain

\begin{lemma}[{cf.~\cite[Lemma~35]{Gokova}}]
\label{lemma:qivj-curve} One has
$${12\choose n}-\frac{(n+1)d}{2}>q_n(V_{n-1})\geqslant {12\choose n}-nd-1.$$
\end{lemma}

\begin{proof}
Recall that the variety $V_{n-1}\subset\Pi_{n-1}\cong\P^{11-n+1}$
is a smooth curve of degree~$d$, since~$V$ is normal. Recall also
$V_{n-1}$ is irreducible, since $V$ is irreducible. Let $g$ be the
genus of the curve~$V_{n-1}$. Then it follows from the adjunction
formula that
$$
2g-2=(K_{V}+(n-1))\cdot H^{n-1}=K_{V}\cdot H^{n-1}+(n-1)d<(n-1)d,
$$
since $K_{V}\cdot H^{n-1}<0$, because $-K_{V}$ is big. On the
other hand, we have
$$
q_m(V_{n-1})={11-n+1+m\choose
m}-h^{0}\Big(\mathcal{O}_{V_{n-1}}\big(mH_{n-1}\big)\Big)%
$$
for every $m\geqslant n$, because
$h^{1}(\mathcal{O}_{\Pi_{n-1}}(m)\otimes \mathcal{I}_{V_{n-1}})=0$
for every $m\geqslant n-1$. Since $2g-2<nd$, the divisor
$nH_{n-1}$ is non-special. Therefore, it follows from the
Riemann--Roch theorem that
$$
q_n(V_{n-1})={12\choose n}-nd+g-1,%
$$
which implies the required inequalities, since $2g-2<(n-1)d$ and
$g\geqslant 0$.
\end{proof}

Combining Lemma~\ref{lemma:qivj-curve} and
Corollary~\ref{corollary:qivj-stupid}, we obtain

\begin{corollary}\label{corollary:curve}
If $n=9$, then
$$
\max\big(0, 219-9d\big)\le q_9-8q_8+28q_7-56q_6 <220-5d.
$$
\end{corollary}

As a by-product of Corollary~\ref{corollary:curve}, we get

\begin{corollary}\label{corollary:degree-bounds}
If $n=9$, then $1\le d\le 43$.
\end{corollary}

The above restrictions reduce the problem to a combinatorial
question of finding all polynomials $H_V$ of degree $n$ with a
leading coefficient~\mbox{$d/n!$}, such that $h_m=H_V(m)\in\Sigma_m$ for
sufficiently many $m\geqslant 1$, and such that the numbers $h_m$,
$q_m=h^0(\OO_{\P^{11}}(m))-h_m$ and $d$ satisfy the conditions arising
from
Corollaries~\ref{corollary:curve}
and~\ref{corollary:degree-bounds}. This can be
done in a straighforward way, although the number of cases to be
considered is so large that
it requires some checks to be done by a computer.
Doing this, we get the following
facts which we leave without proofs.

\begin{lemma}\label{lemma:python-8}
There are no polynomials $\mathrm{H}(m)$ of degree $n\le 8$ such
that the values $h_m=\mathrm{H}(m)$ are in $\Sigma_m$ for $1\le
m\le 12$.
\end{lemma}

\begin{lemma}\label{lemma:python-9}
There does not exist a polynomial $\mathrm{H}(m)$ of degree $n=9$
with a leading coefficient~$d/n!$ with $d\in\Z$ and $1\le
d\leqslant 43$, such that the values $h_m=\mathrm{H}(m)$ are in
$\Sigma_m$ for $1\le m\le 12$, and the numbers $h_m$ and
$q_m={11+m\choose m}-h_m$ satisfy the bounds of
Corollary~\ref{corollary:curve}.
\end{lemma}

This completes the proof of Theorem~\ref{theorem:Suzuki}.

\appendix
\section{Semi-invariants of low degrees}
\label{section:semi-invariants}

In this section we list the results
of the computations of the low degree semi-invariants
of the irreducible representations of the relevant groups
(communicated to us by A.\,Zavarnitsyn).
The GAP software (see \cite{GAP}) was used to carry
them out.

The tables below contain the information about the representations
of stem extensions of sporadic simple groups with the least possible
value of $\mu(U)=d(U)/\dim(U)$ among all irreducible representations
$U$ of the corresponding group $G$, where $d(U)$ is a minimal degree of
a semi-invariant\footnote{Actually, for stem extensions of simple
groups all semi-invariants are invariants.}
of $G$ for the $G$-representation $U$.
In each case we
list the values of $d(U)$ and $\dim(U)$ for which the minimum of
$\mu(U)$ is attained (except for the groups $2.\HaJ$ and $6.\Suz$
where we list slightly different information, see below).
It appears a posteriori that for each of our groups
the value of $\dim(U)$~--- and thus also of $d(U)$~--- giving the minimal
value of~$\mu(U)$ is unique.

\medskip
Mathieu groups.

\begin{tabular}{|c||c||c|c||c|c|c|c|c|c||c||c|}
\hline
$G$ &$\MM_{11}$& $\MM_{12}$& $2.\MM_{12}$&
$\MM_{22}$& $2.\MM_{22}$& $3.\MM_{22}$& $4.\MM_{22}$& $6.\MM_{22}$&
$12.\MM_{22}$&
$\MM_{23}$&  $\MM_{24}$\\
\hline
$d(U)$ & $4$  & $3$  & $6$  & $2$  & $4$  & $3$  & $4$  & $6$  & $12$
& $2$  & $4$\\
\hline
$\dim(U)$ & $10$ & $16$ & $10$ & $21$ & $10$ & $21$ & $56$ & $66$ & $120$
& $22$ & $45$\\
\hline
\end{tabular}

\medskip
Conway groups.

\begin{tabular}{|c||c|c||c||c|}
\hline
$G$ & $\Co_1$& $2.\Co_1$& $\Co_2$& $\Co_3$\\
\hline
$d(U)$ & $2$   & $2$  & $2$  & $2$\\
\hline
$\dim(U)$ & $276$ & $24$ & $23$ & $23$\\
\hline
\end{tabular}

\medskip
Leech lattice groups except Conway groups (for the group $2.\HaJ$
the $6$-dimesional representations are ignored, and for the group
$6.\Suz$ the $12$-dimensional representations are ignored when
computing $d$ and $n$~--- these actually lead to
weakly-exceptional singularities, and the values of~$d$ and~$n$
with $d/n\ge 1$).

\begin{tabular}{|c||c|c||c|c||c|c||c|c|c|c|}
\hline
$G$ &$\HS$ & $2.\HS$& $\HaJ$ & $2.\HaJ${}\footnotemark &
$\McL$& $3.\McL$& $\Suz$ & $2.\Suz$ &
$3.\Suz$ & $6.\Suz${}\footnotemark\\
\hline
$d(U)$ & $2$  & $2$  & $2$  & $4$ & $2$  & $3$
& $2$   & $4$   & $6$ & $6$\\
\hline
$\dim(U)$ & $22$ & $56$ & $14$ & $14$ & $22$ & $126$
& $143$ & $220$ & $78$ & $780$\\
\hline
\end{tabular}

\addtocounter{footnote}{-1}
\footnotetext{Without $6$-dimensional representations.}
\addtocounter{footnote}{1}
\footnotetext{Without $12$-dimensional representations.}

\medskip
Fischer groups.

\begin{tabular}{|c||c|c|c|c||c||c|c|}
\hline
$G$ & $\Fi_{22}$& $2.\Fi_{22}$& $3.\Fi_{22}$&
$6.\Fi_{22}$& $\Fi_{23}$& $\Fi_{24}'$& $3.\Fi_{24}'$\\
\hline
$d(U)$ & $2$  & $2$   & $3$   & $6$    & $2$   & $2$  & $3$\\
\hline
$\dim(U)$ & $78$ & $352$ & $351$ & $1728$ & $782$ & $8671$ & $783$\\
\hline
\end{tabular}

\medskip
Other Monster sections.

\begin{tabular}{|c||c||c||c||c|c||c|}
\hline
$G$ &$\He$ & $\HN$& $\Th$&
$\B$ & $2.\B$& $\M$\\
\hline
$d(U)$ & $3$  & $2$   & $2$   & $2$    & $2$     & $2$\\
\hline
$\dim(U)$ & $51$ & $133$ & $248$ & $4371$ & $96256$ & $196883$\\
\hline
\end{tabular}

\medskip
Tits group and pariahs.
\nopagebreak

\begin{tabular}{|c||c||c||c|c||c|c||c|c||c||c|}
\hline
$G$ &$\T$ & $\J_1$& $\ON$&
$3.\ON$& $\J_3$& $3.\J_3$& $\Ru$& $2.\Ru$& $\J_4$& $\Ly$\\
\hline
$d(U)$ & $6$  & $2$  & $4$     & $6$   & $3$  & $6$  & $4$   & $4$  & $4$
& $6$ \\
\hline
$\dim(U)$ & $26$ & $56$ & $13376$ & $342$ & $85$ & $18$ & $378$ & $28$ & $1333$
& $2480$ \\
\hline
\end{tabular}

\end{document}